\documentclass[11pt]{article}
\usepackage[letterpaper]{geometry}
\usepackage{amsmath,amsthm,amsfonts,amssymb,booktabs,mathrsfs,bm}
\usepackage{enumerate,color,xcolor}
\usepackage{enumitem}
\usepackage{graphicx}
\usepackage{subfigure}
\usepackage{url}
\usepackage{todo}
\usepackage[T1]{fontenc}
\usepackage[utf8]{inputenc}
\usepackage[colorlinks,linkcolor=blue,citecolor=orange]{hyperref}
\numberwithin{equation}{section}
\theoremstyle{plain}
\newtheorem{theorem}{Theorem}[section]
\newtheorem{lemma}{Lemma}[section]

\newtheorem{proposition}{Proposition}[section]

\newtheorem{corollary}{Corollary}[section]
\newtheorem{remark}{Remark}
\begin{document}

\title{ Moderate Deviations for Random-Indexed Cluster Counts in Hierarchical Pitman–Yor Models}

\author{\bf Nian Yao \thanks{ College of Mathematics Science, Shenzhen University, Shenzhen, Guangdong Province, 518060, China. Email: yaonian@szu.edu.cn. ORCID: 0000-0002-9233-3384.}}

\date{\today}
\maketitle

\begin{abstract}
Consider a sample of size $N$ from a two-layer hierarchical
Pitman--Yor model, and let $\xi(N)$ denote the random number of
clusters generated at the first layer of the hierarchy.
We study moderate deviation principles for partition statistics
evaluated at the random sample size $\xi(N)$. In particular, we
establish moderate deviation principles for the total number of
clusters $K_{\xi(N)}$ and for the frequency count
$M_{l,\xi(N)}$, the number of upper-level clusters represented by
exactly $l$ first-level clusters. Our analysis combines fixed-sample-size small-tilt asymptotics with
a moderate deviation principle for the random index $\xi(N)$.
A uniform small-tilt argument, together with exponentially weighted
tail estimates, allows us to pass from deterministic sample sizes to
the hierarchical random-index setting. We obtain explicit deviation
speeds and good rate functions on a family of intermediate scales
between the typical growth order $N^{\alpha_1\alpha_2}$ and the
linear scale $N$. The resulting rate functions exhibit a common
power-law structure governed by the product
$\alpha_1\alpha_2$, revealing a multiplicative effect of the two
levels of the hierarchy.
\end{abstract}



%%%%%%%%%%%%%%%%%%%%%%%%%%%%%%%%%%%%
\section{Introduction}
Consider a population of a countable number of individuals belonging to a possibly infinite number of types. The type of each individual is labelled by a point in a Polish space $S$. The type proportions in the population are thus a point
\begin{equation}
\mathbf{p} = (p_i)_{i \ge 1} \in \Delta := \left\{ \mathbf{q} = (q_i)_{i \ge 1} : q_i \ge 0,\ \sum_{j \ge 1} q_j = 1 \right\}.
\end{equation}

For each $n \ge 1$, let $(X_1, \ldots, X_n)$ be a random sample from the population, with $X_i$ denoting the type of the $i$-th sample point. Clearly, the sample induces a random partition of $\{1, \ldots, n\}$ whose blocks correspond to the equivalence classes for the random equivalence relation
\begin{equation}
i \sim j \quad \Longleftrightarrow \quad X_i = X_j \quad \text{a.s.}
\end{equation}
In such a framework, for a sample of size $n$, two important statistics for measuring sample diversity are:
\begin{itemize}
\item[(i)] the number $K_n$ of different types in the sample;
\item[(ii)] the number $M_{l,n}$ of different types with frequency $1 \le l \le n$ in the sample.
\end{itemize}

We study moderate deviations for random-indexed functionals of random partitions
arising from hierarchical species sampling models, a setting that lies at the
intersection of Bayesian nonparametric and large deviation theory.

For each $1 \le i \le d$, let $K_{i,N_i}$ denote the number of clusters
in a sample of size $N_i$, defined analogously to $K_n$ above. Let $\xi(N)$ denote the random index corresponding to the total number
of clusters generated in the first layer of the hierarchy (see Section~\ref{Random index representation}
for the precise definition). The quantities of interest, such as $K_{\xi(N)}$ and $M_{l,\xi(N)}$,
are thus evaluated at a random index rather than a deterministic sample size. This random-index structure introduces significant technical challenges. In particular, the sequence $\{\xi(N)\}$ is not a stopping time, and the composition $K_{\xi(N)}$ does not preserve the Markov property of $(K_n)$. As a consequence, standard approaches based on temporal dependence,
generators, or Markov semigroups are not applicable.

The asymptotic behaviour of the occupancy counts $K_n$ and $M_{l,n}$
has been extensively studied in the literature, with fundamentally
different regimes depending on the parameter $\alpha$. In the classical case $\alpha=0$, corresponding to the Dirichlet
distribution, $K_n$ admits a representation as a sum of independent
Bernoulli random variables. This structure leads to precise limit
theorems, including a strong law of large numbers and a central limit
theorem under logarithmic scaling. These results can be traced back to
the seminal work of Goncharov~\cite{Goncharov1944}, and were further
developed by Karlin~\cite{Karlin1967} and subsequent works such as
\cite{KorwarHollander1973, Hansen1990}. In contrast, when $\alpha \in (0,1)$, the asymptotic behaviour changes substantially and the Gaussian approximation breaks down. In this
regime, $K_n$ exhibits polynomial growth of order $n^\alpha$, and its
fluctuations are governed by a non-Gaussian limit. Pitman~\cite{Pitman2006}
established that $n^{-\alpha}K_n$ converges almost surely to a
scaled Mittag--Leffler random variable, often referred to as the
$\alpha$-diversity. This result also implies corresponding limits for
the occupancy counts $M_{l,n}$. These two regimes highlight a fundamental transition from logarithmic to polynomial scaling, and from Gaussian to non-Gaussian fluctuations.

Large and moderate deviation principles for occupancy-type statistics
with deterministic sample size have also been established in various settings,
providing refined asymptotic descriptions beyond the law of large numbers
and central limit regime. In particular, large deviation principles for Poisson--Dirichlet distributionsand related occupancy statistics have been studied by \cite{Feng2010},
 who derived precise exponential asymptotics for cluster counts under
Poisson--Dirichlet sampling models. Earlier contributions in the context of random combinatorial structures and population genetics include \cite{FengHoppe1998}, where large deviation principles were established for Ewens-type sampling formulas and related
partition structures. From the perspective of occupancy problems with infinitely many boxes,
\cite{GnedinHansenPitman2007} provides a detailed analysis of asymptotic
properties of occupancy counts, highlighting the role of power-law behaviour
and connections with exchangeable random partitions. Moderate deviations for Ewens--Pitman sampling models were subsequently
studied in \cite{FavaroFengGao2018}. In particular, moderate deviation
principles were established for the number $K_n$ of distinct types observed
in a sample and for the frequency counts $M_{l,n}$, the number of distinct
types appearing exactly $l$ times. These results provide the
fixed-sample-size moderate deviation asymptotics that form the starting point
for the random-index analysis developed in the present paper.

More recently, \cite{FavaroFengPaguyo2025} studied the asymptotic behavior
of cluster counts in hierarchical species sampling models. In particular,
their random-index representation reduces hierarchical cluster statistics
to single-layer partition statistics evaluated at a random sample size, and
they established, among other results, large deviation principles at the
macroscopic scale.

The moderate-deviation regime for such random-indexed statistics, however,
has not been addressed. In particular, the fixed-$\lambda$ random-index
large-deviation framework does not directly capture the vanishing tilts that
arise at moderate-deviation scales. To the best of our knowledge, moderate
deviation principles for quantities of the form
$K_{\xi(N)}$ and $M_{l,\xi(N)}$
have not previously been established in this hierarchical setting.

In this paper, we develop a two-layer moderate deviation analysis for the
hierarchical model. We first derive a moderate deviation principle for the
random index $\xi(N)$ and then combine it with fixed-sample-size small-tilt
asymptotics. A uniform small-tilt argument allows us to handle the random
sample size and yields explicit moderate deviation principles for
$K_{\xi(N)}$ and $M_{l,\xi(N)}$. The resulting rate functions exhibit a
multiplicative structure in the deviation exponents and describe
intermediate-scale fluctuations between the typical growth regime and the full large-deviation regime. The contributions of our paper can be summarized as follows.

\begin{itemize}

\item
We establish moderate deviation principles for cluster counts in
two-layer Pitman--Yor models with general admissible concentration
parameters and finitely many groups having positive limiting sample
proportions. The resulting family of moderate-deviation scales fills
the intermediate regime between the typical growth order
$N^{\alpha_1\alpha_2}$ and the linear scale $N$ of the full
large-deviation principle.

\item
We obtain explicit deviation speeds and good rate functions for both
the total cluster count $K_{\xi(N)}$ and the count $M_{l,\xi(N)}$ of
upper-level clusters represented by exactly $l$ first-level clusters.
In both cases, the rate functions have a power-law exponent determined
by the product $\alpha_1\alpha_2$, while the frequency count carries an
additional explicit $l$-dependent factor. Under the canonical
normalization adopted here, the rate functions are independent of the
concentration parameters and of the limiting sample proportions of the
groups.

\item
At the methodological level, we develop a random-index small-tilt
argument that is not covered by the fixed-$\lambda$ large-deviation
composition used in the existing hierarchical theory. More precisely,
we establish uniform small-tilt asymptotics over compact ranges of the
random index together with exponentially weighted tail estimates,
which justify passing from fixed-sample-size log-moment-generating-
function asymptotics to the hierarchical random-index setting.
Combining these estimates with the moderate deviation principle for
$\xi(N)$ yields the two-layer limiting log-moment generating functions
and the multiplicative exponent $\alpha_1\alpha_2$.

\end{itemize}

The remainder of the paper is organized as follows.
In Section~2, we introduce the hierarchical model, the random-index
representation, and the notation used throughout the paper.
Section~3 recalls the relevant single-layer moderate deviation results and
establishes the moderate deviation principle for the random index $\xi(N)$.
In Section~4, we develop the uniform small-tilt analysis and prove the
two-layer moderate deviation principles for the random-indexed cluster statistics.
%%%%%%%%%%%%%%%%%%%%%%%%%%%%%%%%%%%%%

\section{Preliminaries}
\subsection{Notation and Definitions}

For two sequences $(a_n)$ and $(b_n)$, we write $a_n \asymp b_n$ if
$\lim_{n\to\infty} a_n/b_n \in (0,\infty)$, and $a_n \sim b_n$ if the limit equals $1$.
We also write $a_n = O(b_n)$ if $a_n \le c b_n$ for some $c>0$ and all large $n$, and
$a_n = o(b_n)$ if $a_n/b_n \to 0$.

A measurable function $f:(0,\infty)\to(0,\infty)$ is regularly varying with index
$\sigma \in \mathbb{R}$ if
\begin{equation}
f(x) = x^\sigma L(x),
\end{equation}
where $L$ is slowly varying, i.e.
\begin{equation}
\frac{L(ax)}{L(x)} \to 1 \quad \text{as } x\to\infty,\quad \forall a>0.
\end{equation}
%%%%%%%%%%%%%%%%%%%%%%%%%%%%%%%%%%%%%%

\subsection{Random index representation}
\label{Random index representation}

Fix $d\geq1$, $\alpha,\beta\in(0,1)$,
$\theta_0>-\alpha$, and $\theta>-\beta$.
Let $\nu$ be a diffuse probability measure on a Polish
space $S$. For deterministic positive integers
$N_1,\ldots,N_d$, write
\[
N=\sum_{i=1}^d N_i.
\]

We use the latent partition construction of the hierarchical
Pitman--Yor model; see \cite{CamerlenghiEtAl2019,FavaroFengPaguyo2025}.  For each $1\leq i\leq d$, let
$K_{i,N_i}(\beta,\theta)$ denote the number of latent
first-level clusters in group $i$, whose partition follows
the Ewens--Pitman model with parameters $(\beta,\theta)$.
These cluster counts are independent; see
Camerlenghi et al.~\cite[Theorem~6 and Remark~1]{CamerlenghiEtAl2019}
and also Section~2.2 of \cite{FavaroFengPaguyo2025}.

Define the random index
\begin{equation}
\xi(N):=\sum_{i=1}^d K_{i,N_i}(\beta,\theta),
\end{equation}
which is the total number of latent first-level clusters.

Independently of the first-level partitions, consider an
Ewens--Pitman partition with parameters $(\alpha,\theta_0)$,
equivalently the sample partition induced by
$\mathrm{PYP}(\alpha,\theta_0,\nu)$.
Let $K_n(\alpha,\theta_0)$ denote its number of blocks
and let $M_{r,n}(\alpha,\theta_0)$ denote its number of
blocks of size $r$, for a fixed integer $r\geq1$.

In the hierarchical construction, the $\xi(N)$ latent
first-level clusters are partitioned into upper-level
clusters according to this independent partition.
Let $K_N(\alpha,\theta_0;\beta,\theta)$ denote the total
number of upper-level clusters, and let
$M_{r,N}(\alpha,\theta_0;\beta,\theta)$ denote the number
of upper-level clusters containing exactly $r$ latent
first-level clusters.

Under this common construction, the random index representations established
in Camerlenghi et al.~\cite[Theorem~6 and Remark~1]{CamerlenghiEtAl2019}
and recalled in Section~2.2 of \cite{FavaroFengPaguyo2025} yield
\begin{equation}
K_N(\alpha,\theta_0;\beta,\theta)
\overset{\mathrm{a.s.}}{=}
K_{\xi(N)}(\alpha,\theta_0),
\end{equation}
and
\begin{equation}
M_{r,N}(\alpha,\theta_0;\beta,\theta)
\overset{\mathrm{a.s.}}{=}
M_{r,\xi(N)}(\alpha,\theta_0).
\end{equation}

Thus, the hierarchical cluster statistics can be studied
through single-layer partition statistics evaluated at
the independent random index $\xi(N)$.

%%%%%%%%%%%%%%%%%%%%%%%%%%%%%%%%%%%%%%%

\section{Single-layer moderate deviations and the random index}

We establish moderate deviation principles (MDPs) for the
randomly indexed quantities $K_{\xi(n)}$ and $M_{l,\xi(n)}$.
In contrast to the classical setting with deterministic sample size,
the presence of the random index $\xi(n)$ induces a coupling between
the deviation behaviour of the index and that of the underlying
occupancy statistics. This leads to a two-layer deviation structure, where the resulting
rate function is determined by the composition of the corresponding
log-Laplace asymptotics.

\subsection{Moderate deviation results for the single-layer model }
To analyze the two-layer structure arising from the random index,
we first establish the corresponding moderate deviation results
for the single-layer setting, namely for occupancy statistics
with deterministic sample size. The following two lemmas provide
these fundamental ingredients.

Specifically, we assume that $(\beta_n)_{n \geq 1}$ is a sequence such that
\begin{align}\label{beta_n}
\lim_{n \to +\infty} \frac{\beta_n}{n^{1-\alpha_1}} = 0 \quad \text{and} \quad \lim_{n \to +\infty} \frac{\beta_n}{(\ln n)^{1-\alpha_1}} = +\infty. 
\end{align}
The assumption that $\beta_n$ grows faster than $(\ln n)^{1-\alpha_1}$ is crucial for establishing MDPs for the sequences $(K_n)_{n \geq 1}$ and $(M_{r,n})_{n \geq 1}$. Throughout the rest of the paper, for any $a \in \mathbb{R}$ and for any $n \in \mathbb{N}_0$ we denote by $(a)_{n \uparrow 1}$ the rising (or ascending) factorial, i.e., 
\[
(a)_{n \uparrow 1} := a(a+1)\cdots(a+n-1)
\]
with the proviso $(a)_0 := 1$.

\begin{lemma}[Theorem 1 in \cite{FavaroFengGao2018}]\label{lem:FFG1}
For $\alpha_1 \in (0,1)$ and $\theta > -\alpha_1$, the sequence $((\beta_n n^{\alpha_1})^{-1} K_n)_{n \geq 1}$ satisfies a large deviation principle on $\mathbb{R}$ with speed 
$\beta_n^{\frac{1}{1-{\alpha_1}}}$ and rate function $I_{\alpha_1}(\cdot)$ defined by

\begin{equation}\label{K rate}
I_{\alpha_1}(x)=
\begin{cases}
(1-{\alpha_1})\alpha_1^{\!\frac{\alpha_1}{1-{\alpha_1}}}
\,x^{\frac{1}{1-{\alpha_1}}}, & \text{if }x\ge0,\\[1.25em]
+\infty, & \text{if } x<0.
\end{cases}
\end{equation}
\end{lemma}

\begin{lemma}[Theorem 2 in \cite{FavaroFengGao2018}]
\label{lem:FFG2}
For $\alpha_1 \in (0,1)$ and $\theta > -\alpha_1$, the sequence $((\beta_n n^{\alpha_1})^{-1} M_{l,n})_{n \geq 1}$ satisfies a large deviation principle on $\mathbb{R}$ with speed 
$\beta_n^{\frac{1}{1-{\alpha_1}}}$ and rate function $I_{\alpha_1}(\cdot)$ defined by

\begin{equation}\label{M rate}
I_{\alpha_1,l}(x)=
\begin{cases}
\displaystyle
\frac{1-\alpha_1}{\alpha_1}
\left(
\frac{l!}{(1-\alpha_1)_{(l-1)\uparrow1}}
\right)^{\frac{1}{1-\alpha_1}}
x^{\frac{1}{1-\alpha_1}},
& \text{if }x\ge0,\\[1.25em]
+\infty,
& \text{if } x<0.
\end{cases}
\end{equation}
\end{lemma}

\begin{remark}
The explicit rate function in \eqref{M rate} is obtained
by taking the Legendre transform of the limiting log-mgf
derived in the proof of Theorem~2 of
\cite{FavaroFengGao2018}; see equations (2.5), (2.10),
and (2.11) therein.
\end{remark}

%%%%%%%%%%%%%%%%%%%%%%%%%%%%%%

\subsection{Moderate Deviation Principle for \(\xi(N)/N^r\) via the Contraction Principle}
Let \(\xi(N) = \sum_{i=1}^d K_{i,N_i}(\alpha_2,\theta)\), where each \(K_{i,N_i}\) denotes the number of clusters in a sample of size \(N_i\), and \(N_i \sim w_i N\) for some weights \(w_i > 0\) satisfying \(\sum_{i=1}^d w_i = 1\). We want derive the moderate deviation for \(\xi(N)/N^r\)with $\alpha_2<r<1$.

Specifically, we also assume that $(\beta_n)_{n\ge1}$ satisfies
\begin{align}\label{beta_n_alpha2}
\lim_{n\to\infty}\frac{\beta_n}{n^{1-\alpha_2}}=0
\quad\text{and}\quad
\lim_{n\to\infty}\frac{\beta_n}{(\ln n)^{1-\alpha_2}}=+\infty.
\end{align}
This condition is used to apply the one-dimensional moderate deviation
principle to the cluster counts $K_{i,N_i}(\alpha_2,\theta)$ in the first
layer, and hence to derive the LDP for the random index
\[
\xi(N)=\sum_{i=1}^d K_{i,N_i}(\alpha_2,\theta).
\]

%%%%%%%%%%%%%%%%%%%%%%%%%%%%%%%%%%%%%%%%

\begin{theorem}[LDP for the weighted sum]\label{thm:sum} Assume $\beta_n$ satisfied equation (\ref{beta_n_alpha2}). Fix $\alpha_2\in(0,1)$, for $i=1,\dots,d$, define
\[
Z_{i,N}:=\frac{K_{i,N_i}(\alpha_2,\theta)}
{\beta_{N_i}N_i^{\alpha_2}},
\qquad
s_{i,N}:=\beta_{N_i}^{\frac{1}{1-\alpha_2}}.
\]
Assume that $K_{1,N_1},\dots,K_{d,N_d}$ are independent. Let
\[
v_N:=\sum_{i=1}^d s_{i,N},
\qquad
\gamma_{i,N}:=\frac{s_{i,N}}{v_N}\to\gamma_i\in(0,1),
\qquad
\sum_{i=1}^d\gamma_i=1.
\]
For weights $a_{i,N}\to a_i>0$, set
\[
T_N:=\sum_{i=1}^d a_{i,N}Z_{i,N}.
\]
Then $\{T_N\}$ satisfies an LDP on $\mathbb R$ with speed $v_N$ and good rate
function
\[
I_T(x)=
\begin{cases}
\displaystyle
c_{\alpha_2}B_T^{-\frac{\alpha_2}{1-\alpha_2}}
x^{\frac{1}{1-\alpha_2}},&x\ge0,\\[1.2ex]
+\infty,&x<0,
\end{cases}
\]
where
$
c_{\alpha_2}:=(1-\alpha_2)\alpha_2^{\frac{\alpha_2}{1-\alpha_2}}
$
and
$
B_T:=
\sum_{i=1}^d
a_i^{\frac{1}{\alpha_2}}\gamma_i^{-\frac{1-\alpha_2}{\alpha_2}}.
$
\end{theorem}

\begin{proof} By Lemma~\ref{lem:FFG1}, applied with parameter $\alpha_2$ and sample size $N_i$, each $Z_{i,N}$ satisfies an LDP with speed $s_{i,N}$ and rate function \[ I_{\alpha_2}(x)= \begin{cases} c_{\alpha_2}x^{\frac{1}{1-\alpha_2}},&x\ge0,\\ +\infty,&x<0. \end{cases} \] Since $s_{i,N}/v_N\to\gamma_i$, the same coordinate LDP, rewritten with the common speed $v_N$, has rate function $\gamma_i I_{\alpha_2}$. Using the independence of the coordinates, the finite-dimensional product LDP gives that \[ Z_N:=(Z_{1,N},\dots,Z_{d,N}) \] satisfies an LDP on $\mathbb R^d$ with speed $v_N$ and rate function \[ I_{\mathrm{vec}}(x_1,\dots,x_d) = \sum_{i=1}^d \gamma_i I_{\alpha_2}(x_i). \] Since $a_{i,N}\to a_i$, the maps $f_N(x)=\sum_{i=1}^d a_{i,N}x_i $ converge uniformly on compact sets to $f(x)=\sum_{i=1}^d a_i x_i. $ Therefore, by the contraction principle, \[ I_T(x) = \inf\left\{ \sum_{i=1}^d \gamma_i c_{\alpha_2}x_i^{\frac{1}{1-\alpha_2}}: x_i\ge0,\ \sum_{i=1}^d a_i x_i=x \right\}. \] Solving this variational problem by H\"older's inequality yields \[ I_T(x) = c_{\alpha_2}B_T^{-\frac{\alpha_2}{1-\alpha_2}}
x^{\frac{1}{1-\alpha_2}}, \qquad x\ge0. \]For $x<0$, the constraint cannot be satisfied with $x_i\ge0$, hence $I_T(x)=+\infty$. \end{proof}

%%%%%%%%%%%%%%%%%%%%%%%%%%%%%%%%%%%%%%

\begin{corollary}[LDP for $\xi(N)/(\beta_N N^{\alpha_2})$]\label{cor:MDP-newZ}
Under the notation and assumptions of Theorem~\ref{thm:sum}.  Assume further that for all $1\le i\le d$, $\frac{\beta_{N_i}}{\beta_N}\to u_i$. Let 

\begin{equation}\label{TN0}
T_N^{(0)}:=\frac{\xi(N)}{\beta_N N^{\alpha_2}}
=\sum_{i=1}^d a_{i,N}\,Z_{i,N}, \end{equation}
we have that $\{T_N^{(0)}\}$ has an LDP with speed $v_N=\sum_i \beta_{N_i}^{\frac{1}{1-\alpha_2}}$
and good rate funtion 
\[
I_{T^{(0)}}(x)=
\begin{cases}
\displaystyle
\frac{(1-\alpha_2)\alpha_2^{\frac{\alpha_2}{1-\alpha_2}}
}{\sum_{j=1}^d u_j^{\frac{1}{1-\alpha_2}}}x^{\frac{1}{1-\alpha_2}},
& x\ge 0,\\[1.2ex]
+\infty, & x<0.
\end{cases}
\]
\end{corollary}

\begin{proof}
Set
$
p:=\frac{1}{1-\alpha_2}, 
c_{\alpha_2}:=(1-\alpha_2)
\alpha_2^{\frac{\alpha_2}{1-\alpha_2}},
$
and
$
S:=\sum_{j=1}^d u_j^{\frac{1}{1-\alpha_2}}.
$
We first prove the LDP for
$
T_N^{(0)}
=
\frac{\xi(N)}{\beta_NN^{\alpha_2}}.
$
By Theorem~\ref{thm:sum}, $\{T_N^{(0)}\}$ satisfies an LDP with speed $v_N$ and good rate function
\[
I_{T^{(0)}}(x)
=
\begin{cases}
\displaystyle
c_{\alpha_2}
\left(
\sum_{i=1}^d
a_i^{\frac{1}{\alpha_2}}
\gamma_i^{-\frac{1-\alpha_2}{\alpha_2}}
\right)^{-\frac{\alpha_2}{1-\alpha_2}}
x^{\frac{1}{1-\alpha_2}},
& x\ge0,\\[1.2ex]
+\infty,&x<0.
\end{cases}
\]
Since
$
a_i=u_iw_i^{\alpha_2},
\gamma_i=\frac{u_i^{\frac{1}{1-\alpha_2}}}{S},
$
we obtain
$
\sum_{i=1}^d
a_i^{\frac{1}{\alpha_2}}
\gamma_i^{-\frac{1-\alpha_2}{\alpha_2}}=
S^{\frac{1-\alpha_2}{\alpha_2}}.
$
Here we also know that
$$\gamma_{i,N}=\frac{\beta_{N_i}^{\frac{1}{1-\alpha_2}}}{\sum_{j=1}^d \beta_{N_j}^{\frac{1}{1-\alpha_2}}} \longrightarrow \gamma_i=\frac{u_i^{\frac{1}{1-\alpha_2}}}{\sum_{j=1}^d u_j^{\frac{1}{1-\alpha_2}}},
$$ and 
$$a_{i,N}:=\frac{\beta_{N_i} N_i^{\alpha_2}}{\beta_N N^{\alpha_2}}\ \longrightarrow\ a_i=u_iw_i^{\alpha_2}.$$
Therefore
\[
I_{T^{(0)}}(x)
=
\begin{cases}
\displaystyle
\frac{c_{\alpha_2}}{S}
x^{\frac{1}{1-\alpha_2}},&x\ge0,\\[1ex]
+\infty,&x<0.
\end{cases}
\]
\end{proof}

\begin{proposition}[MDP for the random index $\xi(N)/N^{r}$]
\label{prop:MDP-xi}
Let $\xi(N):=\sum_{i=1}^d K_{i,N_i}(\alpha_2,\theta)$, where $K_{1,N_1},\ldots,K_{d,N_d}$ are independent and
$
\frac{N_i}{N}\longrightarrow w_i\in(0,1)$,
$\sum_{i=1}^d w_i=1.
$
Fix $r\in(\alpha_2,1)$ and define
\[
W_N:=\frac{\xi(N)}{N^r},
\qquad
s_N:=N^{\frac{r-\alpha_2}{1-\alpha_2}}.
\]
Then $\{W_N\}_{N\ge1}$ satisfies an LDP on $[0,\infty)$ with speed
$s_N$ and good rate function
\[
I_W(x)=
\begin{cases}
\displaystyle
(1-\alpha_2)
\alpha_2^{\frac{\alpha_2}{1-\alpha_2}}x^{\frac{1}{1-\alpha_2}},&x\ge0,\\[1ex]
+\infty,&x<0.
\end{cases}
\]
Equivalently, this is a moderate deviation principle for
$\xi(N)$ at level $N^r$.
\end{proposition}

\begin{proof}
We use the single-layer log-mgf asymptotic obtained in
the proof of Theorem~1 of \cite{FavaroFengGao2018}.
For $\alpha\in(0,1)$ and $\theta>-\alpha$, if a positive
sequence $(\beta_n)$ satisfies
\[
\frac{\beta_n}{n^{1-\alpha}}\to0,
\qquad
\frac{\beta_n}{(\log n)^{1-\alpha}}\to\infty,
\]
then, for every fixed $\lambda\in\mathbb{R}$,
\[
\frac{1}{\beta_n^{\frac{1}{1-\alpha}}}
\log\mathbb{E}
\exp\left\{
\lambda n^{-\alpha}
\beta_n^{\frac{\alpha}{1-\alpha}}
K_n(\alpha,\theta)
\right\}
\longrightarrow
\Psi_\alpha(\lambda),
\]
where
\[
\Psi_\alpha(\lambda)=
\begin{cases}
0,&\lambda\leq0,\\
\lambda^{\frac{1}{\alpha}},&\lambda>0.
\end{cases}
\]
Equivalently, setting $\beta_n=v_n^{1-\alpha}$ gives
\begin{equation}
\label{single-layer small-tilt}
\frac{1}{v_n}
\log\mathbb{E}
\exp\left\{
\lambda\left(\frac{v_n}{n}\right)^\alpha
K_n(\alpha,\theta)
\right\}
\longrightarrow
\Psi_\alpha(\lambda),
\end{equation}
provided
\[
\frac{v_n}{n}\to0,
\qquad
\frac{v_n}{\log n}\to\infty.
\]

The same limit holds along deterministic pairs
$(m_j,v_j)$ satisfying
\[
m_j\to\infty,\qquad
\frac{v_j}{m_j}\to0,\qquad
\frac{v_j}{\log m_j}\to\infty.
\]
Indeed, if convergence failed, one could select a
subsequence along which the error is bounded away
from zero and the sample sizes $m_j$ are strictly
increasing. Define a full sequence to equal $v_j$
at these sample sizes and to equal
$\sqrt{m\log(m+1)}$ at all remaining integers $m$.
This sequence satisfies the conditions of
\eqref{single-layer small-tilt}, giving a contradiction.

We now apply this result to the first-level cluster
counts. Throughout the remainder of the proof,
all limits are taken as $N\to\infty$, with $i$ fixed,
and we write $N_i=N_i(N)$.
Set
\[
\kappa=\frac{r-\alpha_2}{1-\alpha_2},
\qquad
s_N=N^\kappa,
\qquad
t_N=\frac{s_N}{N^r}.
\]
Since $\alpha_2<r<1$, we have $0<\kappa<1$ and
\[
t_N
=
N^{-\frac{\alpha_2(1-r)}{1-\alpha_2}}
\longrightarrow0.
\]
For each $i=1,\ldots,d$, define the local speed
\[
v_{i,N}
:=
N_i t_N^{\frac{1}{\alpha_2}}
=
\frac{N_i}{N}s_N.
\]
The corresponding sample size and moderate deviation
scale in the single-layer result are
\[
n=N_i(N),
\qquad
B_{i,N}:=v_{i,N}^{1-\alpha_2}
=N_i^{1-\alpha_2}N^{r-1},
\]
respectively. By $N_i/N\to w_i>0$, we have $N_i\to\infty$,
$\log N_i\sim\log N$, and
\[
\frac{v_{i,N}}{s_N}\to w_i.
\]
Moreover,
\[
\frac{v_{i,N}}{N_i}
=
\frac{s_N}{N}
=
N^{\kappa-1}
\longrightarrow0,
\qquad
\frac{v_{i,N}}{\log N_i}
=
\frac{N_i}{N}\frac{N^\kappa}{\log N_i}
\longrightarrow\infty.
\]
Equivalently, the corresponding scales satisfy
\[
\frac{B_{i,N}}{N_i^{1-\alpha_2}}
=N^{r-1}\to0,
\qquad
\frac{B_{i,N}}{(\log N_i)^{1-\alpha_2}}
=
\left(
\frac{v_{i,N}}{\log N_i}
\right)^{1-\alpha_2}
\to\infty.
\]

Thus \eqref{single-layer small-tilt}, in its
deterministic-pair form, applies with
$\alpha=\alpha_2$ and $(m_j,v_j)=(N_i(N),v_{i,N})$.
It yields
\[
\frac{1}{v_{i,N}}
\log\mathbb{E}
\exp\left\{
\lambda
\left(\frac{v_{i,N}}{N_i}\right)^{\alpha_2}
K_{i,N_i}(\alpha_2,\theta)
\right\}
\longrightarrow
\Psi_{\alpha_2}(\lambda).
\]
Since
\[
\left(\frac{v_{i,N}}{N_i}\right)^{\alpha_2}
=t_N=\frac{s_N}{N^r},
\qquad
\frac{v_{i,N}}{s_N}\to w_i,
\]
we obtain
\[
\frac{1}{s_N}
\log\mathbb{E}
\exp\left\{
\lambda\frac{s_N}{N^r}
K_{i,N_i}(\alpha_2,\theta)
\right\}
\longrightarrow
w_i\Psi_{\alpha_2}(\lambda).
\]

Finally, by independence of
$K_{1,N_1},\ldots,K_{d,N_d}$ and the identity
$\xi(N)=\sum_{i=1}^dK_{i,N_i}(\alpha_2,\theta)$,
\[
\begin{aligned}
\frac{1}{s_N}
\log\mathbb{E}e^{\lambda s_NW_N}
&=
\sum_{i=1}^d
\frac{1}{s_N}
\log\mathbb{E}
\exp\left\{
\lambda\frac{s_N}{N^r}
K_{i,N_i}(\alpha_2,\theta)
\right\}\\
&\longrightarrow
\sum_{i=1}^d w_i\Psi_{\alpha_2}(\lambda)
=
\Psi_{\alpha_2}(\lambda).
\end{aligned}
\]
The limiting log-mgf is finite and continuously
differentiable on $\mathbb{R}$. Hence the
G\"artner--Ellis theorem yields the stated MDP
with speed $s_N$ and good rate function
\[
I_W(x)
=
\sup_{\lambda\in\mathbb{R}}
\{\lambda x-\Psi_{\alpha_2}(\lambda)\}
=
\begin{cases}
(1-\alpha_2)
\alpha_2^{\frac{\alpha_2}{1-\alpha_2}}x^{\frac{1}{1-\alpha_2}},&x\geq0,\\[1mm]
+\infty,&x<0.
\end{cases}
\]
\end{proof}

%%%%%%%%%%%%%%%%%%%%%%%%%%%%%%%%%%%%%

\section{Two-layer moderate deviations for the hierarchical model}

\begin{lemma}[Uniform small-tilt asymptotics]
\label{lem:uniform-small-tilt}
Fix $a\in(0,1)$ and $\theta>-a$. Let $K_m$ and $M_{l,m}$
denote, respectively, the number of blocks and the number of
blocks of size $l$ in the restriction to $[m]$ of a
$\mathrm{PD}(a,\theta)$ exchangeable random partition, where
$l\geq1$ is fixed. Set $K_0=M_{l,0}=0$.

Let $(A_N)$ and $(s_N)$ be positive deterministic sequences
satisfying
\begin{equation}
\label{eq:small-tilt-scales}
A_N\to\infty,\qquad
s_N\to\infty,\qquad
\frac{s_N}{A_N}\to0,\qquad
\frac{s_N}{\log A_N}\to\infty.
\end{equation}
For $\lambda\in\mathbb{R}$, define
\[
t_N(\lambda)
=
\lambda\left(\frac{s_N}{A_N}\right)^a,
\qquad
\lambda_+=\max\{\lambda,0\},
\]
and
\[
p_a(l)
=
\frac{a}{l!}\prod_{j=1}^{l-1}(j-a),
\]
where the empty product is interpreted as one.

For $X_m=K_m$ or $X_m=M_{l,m}$, write
\[
\Psi_X(\lambda)
=
\begin{cases}
(\lambda_+)^{1/a},
    & X_m=K_m,\\[1mm]
\bigl(p_a(l)\lambda_+\bigr)^{1/a},
    & X_m=M_{l,m}.
\end{cases}
\]
Then, for every fixed $\lambda\in\mathbb{R}$ and $M>0$,
\begin{equation}
\label{eq:uniform-small-tilt}
\lim_{N\to\infty}
\sup_{0\leq y\leq M}
\left|
\frac{1}{s_N}
\log\mathbb{E}
\exp\left\{
t_N(\lambda)X_{\lfloor A_Ny\rfloor}
\right\}
-
y\Psi_X(\lambda)
\right|
=0.
\end{equation}
\end{lemma}

\begin{proof}
We use the single-layer log-mgf limits from the proofs
of Theorems~1 and~2 of \cite{FavaroFengGao2018},
together with the deterministic-pair extension
established above. Specifically, for every fixed
$\eta\in\mathbb{R}$,
\begin{equation}
\label{eq:single-layer-small-tilt}
\frac{1}{w_j}
\log\mathbb{E}
\exp\left\{
\eta\left(\frac{w_j}{m_j}\right)^aX_{m_j}
\right\}
\longrightarrow
\Psi_X(\eta)
\end{equation}
whenever
\begin{equation}
\label{eq:admissible-pairs}
m_j\to\infty,\qquad
\frac{w_j}{\log m_j}\to\infty,\qquad
\frac{w_j}{m_j}\to0.
\end{equation}

Fix $\lambda\in\mathbb{R}$ and $M>0$, and define
\[
F_N^X(y)
=
\frac{1}{s_N}
\log\mathbb{E}
\exp\left\{
t_N(\lambda)X_{\lfloor A_Ny\rfloor}
\right\},
\qquad 0\leq y\leq M.
\]
We first establish convergence along any deterministic
sequence $y_N\in[0,M]$ satisfying $y_N\to y>0$.
Let
\[
m_N=\lfloor A_Ny_N\rfloor,
\qquad
w_N=\frac{m_N}{A_N}s_N.
\]
Then
\[
\frac{m_N}{A_N}\to y,
\qquad
\frac{w_N}{s_N}\to y,
\qquad
\frac{w_N}{m_N}
=
\frac{s_N}{A_N}\to0.
\]
Moreover, since $m_N/A_N\to y>0$, we have
$\log m_N/\log A_N\to1$. Consequently,
\eqref{eq:small-tilt-scales} implies
\[
\frac{w_N}{\log m_N}\to\infty.
\]
Thus $(m_N,w_N)$ satisfies \eqref{eq:admissible-pairs}.
Using the exact identity
\[
t_N(\lambda)
=
\lambda\left(\frac{w_N}{m_N}\right)^a,
\]
we obtain from \eqref{eq:single-layer-small-tilt} that
\begin{equation}
\label{eq:moving-positive-index}
\begin{aligned}
F_N^X(y_N)
&=
\frac{w_N}{s_N}
\left[
\frac{1}{w_N}
\log\mathbb{E}
\exp\left\{
\lambda\left(\frac{w_N}{m_N}\right)^aX_{m_N}
\right\}
\right]\\
&\longrightarrow y\Psi_X(\lambda).
\end{aligned}
\end{equation}

Suppose first that $\lambda\leq0$. Under the common partition
coupling,
\[
0
\leq X_{\lfloor A_Ny\rfloor}
\leq K_{\lfloor A_Ny\rfloor}
\leq K_{\lfloor A_NM\rfloor},
\qquad 0\leq y\leq M.
\]
Since $t_N(\lambda)\leq0$, it follows that
\[
F_N^K(M)\leq F_N^X(y)\leq0,
\qquad 0\leq y\leq M.
\]
By \eqref{eq:moving-positive-index},
\[
F_N^K(M)\longrightarrow M\Psi_K(\lambda)=0.
\]
Since also $\Psi_X(\lambda)=0$, this proves
\eqref{eq:uniform-small-tilt} for $\lambda\leq0$.

Now suppose that $\lambda>0$, and let $y_N\in[0,M]$
satisfy $y_N\to0$. Fix $\delta\in(0,M]$. For all sufficiently
large $N$, we have $y_N\leq\delta$, and hence
\[
0\leq F_N^X(y_N)\leq F_N^K(\delta).
\]
Applying \eqref{eq:moving-positive-index} to the constant
sequence $y_N=\delta$ yields
\[
0
\leq \liminf_{N\to\infty}F_N^X(y_N)
\leq \limsup_{N\to\infty}F_N^X(y_N)
\leq \delta\lambda^{1/a}.
\]
Letting $\delta\downarrow0$, we conclude that
\[
F_N^X(y_N)\longrightarrow0.
\]

Together with \eqref{eq:moving-positive-index}, this shows
that, whenever $y_N\in[0,M]$ and $y_N\to y\in[0,M]$,
\begin{equation}
\label{eq:sequential-small-tilt}
F_N^X(y_N)\longrightarrow y\Psi_X(\lambda).
\end{equation}
It remains to deduce uniform convergence. If
\eqref{eq:uniform-small-tilt} were false, there would exist
$\varepsilon>0$, a subsequence $(N_j)$, and points
$y_j\in[0,M]$ such that
\[
\left|
F_{N_j}^X(y_j)-y_j\Psi_X(\lambda)
\right|
\geq\varepsilon.
\]
By compactness, a further subsequence satisfies
$y_j\to y\in[0,M]$. Applying
\eqref{eq:sequential-small-tilt} along this subsequence and
using the continuity of $y\mapsto y\Psi_X(\lambda)$, we obtain
\[
F_{N_j}^X(y_j)-y_j\Psi_X(\lambda)\longrightarrow0,
\]
a contradiction. This completes the proof.
\end{proof}

%%%%%%%%%%%%%%%%%%%%%%%%%%%%%%%%%%%%%

\begin{lemma}[Exponential tail control for the two-layer
conditional expectation]
\label{lem:two-layer-tail-control}
Fix $\alpha_1,\alpha_2\in(0,1)$, $\theta_0>-\alpha_1$,
and $\theta>-\alpha_2$. Let
$K_{i,m}(\alpha_2,\theta)$, $i=1,\ldots,d$, be the
block-count processes of independent
$\mathrm{PD}(\alpha_2,\theta)$ partitions, where $d$ is fixed.
Suppose that the deterministic sample sizes satisfy
\[
\sum_{i=1}^d N_i=N,
\qquad
\frac{N_i}{N}\to w_i>0,
\qquad
\sum_{i=1}^d w_i=1,
\]
and define
\[
\xi(N)
=
\sum_{i=1}^d K_{i,N_i}(\alpha_2,\theta).
\]
Independently of these partitions, let
$X_m=K_m(\alpha_1,\theta_0)$ or
$X_m=M_{l,m}(\alpha_1,\theta_0)$, where $l\geq1$ is fixed.

Let $r\in(\alpha_2,1)$ and
\[
s_N\sim S N^\kappa,
\qquad
\kappa=\frac{r-\alpha_2}{1-\alpha_2},
\qquad S>0.
\]
Set
\[
W_N=\frac{\xi(N)}{N^r},
\qquad
b_N=N^{r\alpha_1}s_N^{1-\alpha_1},
\]
and
\[
t_N(\lambda)
=
\frac{\lambda s_N}{b_N}
=
\lambda\left(\frac{s_N}{N^r}\right)^{\alpha_1}.
\]
For $\lambda\in\mathbb{R}$, define
\[
H_N^X(\lambda)
=
\mathbb{E}\left[
e^{t_N(\lambda)X_{\xi(N)}}\mid\xi(N)
\right].
\]
Then, for every fixed $\lambda\in\mathbb{R}$,
\begin{equation}
\label{eq:two-layer-tail-control}
\lim_{M\to\infty}\limsup_{N\to\infty}
\frac{1}{s_N}
\log\mathbb{E}\left[
H_N^X(\lambda)\mathbf{1}_{\{W_N>M\}}
\right]
=-\infty.
\end{equation}
Equivalently, the same limit holds with
$H_N^X(\lambda)$ replaced by
$e^{t_N(\lambda)X_{\xi(N)}}$.

More precisely, put
\[
c_\lambda=2(\lambda_+)^{1/\alpha_1},
\qquad
\lambda_+=\max\{\lambda,0\}.
\]
For every $h>c_\lambda$ and $M\geq1$,
\begin{equation}
\label{eq:two-layer-tail-quantitative}
\limsup_{N\to\infty}
\frac{1}{s_N}
\log\mathbb{E}\left[
H_N^X(\lambda)\mathbf{1}_{\{W_N>M\}}
\right]
\leq
S^{(1-\alpha_2)/\alpha_2}h^{1/\alpha_2}
-
(h-c_\lambda)M.
\end{equation}
Here $\log0=-\infty$.
\end{lemma}

\begin{proof}
We use the following sequential form of the single-layer
small-tilt asymptotic, obtained from the log-mgf calculation
in the proof of Theorem 1 of \cite{FavaroFengGao2018}:
for fixed $\alpha\in(0,1)$ and $\vartheta>-\alpha$,
\begin{equation}
\label{eq:tail-single-layer-input}
\frac{
\log\mathbb{E}e^{u_jK_{m_j}(\alpha,\vartheta)}
}{
m_j u_j^{1/\alpha}
}
\longrightarrow1
\end{equation}
whenever
\[
m_j\to\infty,
\qquad
u_j>0,
\qquad
u_j\to0,
\qquad
\frac{m_j u_j^{1/\alpha}}{\log m_j}\to\infty.
\]
For a full sequence indexed by $m$, this follows by taking
$\beta_m=(m u_m^{1/\alpha})^{1-\alpha}$ in the single-layer
log-mgf limit. The formulation for arbitrary pairs follows
by a subsequence argument: a failure would yield a
subsequence with strictly increasing $m_j$; assign the
corresponding speed
$v_{m_j}=m_j u_j^{1/\alpha}$ there and set
$v_m=\sqrt{m\log(m+1)}$ at the remaining integers.
This gives an admissible full sequence and a contradiction.

We first verify the outer exponential-moment limit.
For every fixed $h>0$, the tilt $h s_N/N^r$ tends to zero
and
\[
\frac{
N_i(hs_N/N^r)^{1/\alpha_2}
}{
\log N_i
}
\sim
w_i h^{1/\alpha_2}S^{1/\alpha_2}
\frac{N^\kappa}{\log N}
\longrightarrow\infty.
\]
By independence and \eqref{eq:tail-single-layer-input},
\begin{equation}
\label{eq:tail-outer-logmgf}
\begin{aligned}
\frac{1}{s_N}\log\mathbb{E}e^{h s_NW_N}
&=
\frac{1}{s_N}\sum_{i=1}^d
\log\mathbb{E}
\exp\left\{
\frac{h s_N}{N^r}K_{i,N_i}(\alpha_2,\theta)
\right\}\\
&=
(1+o(1))
\frac{N}{s_N}
\left(\frac{h s_N}{N^r}\right)^{1/\alpha_2}\\
&\longrightarrow
S^{(1-\alpha_2)/\alpha_2}h^{1/\alpha_2}
=:\Gamma(h).
\end{aligned}
\end{equation}

Next suppose that $\lambda>0$, and abbreviate
$K_m=K_m(\alpha_1,\theta_0)$ and
$t_N=t_N(\lambda)$. We claim that
\begin{equation}
\label{eq:tail-uniform-relative-mgf}
\sup_{m\geq\lceil N^r\rceil}
\left|
\frac{
\log\mathbb{E}e^{t_NK_m}
}{
m t_N^{1/\alpha_1}
}
-1
\right|
\longrightarrow0.
\end{equation}
Indeed, $t_N\to0$ and, since $m/\log m$ is increasing
for sufficiently large $m$,
\[
\inf_{m\geq\lceil N^r\rceil}
\frac{m t_N^{1/\alpha_1}}{\log m}
\geq
\frac{
\lambda^{1/\alpha_1}s_N
}{
r\log N
}
\longrightarrow\infty.
\]
If \eqref{eq:tail-uniform-relative-mgf} failed, there
would exist a subsequence $(N_j)$ and integers
$m_j\geq\lceil N_j^r\rceil$ along which the error stays
bounded away from zero. The pairs $(m_j,t_{N_j})$
satisfy all conditions of
\eqref{eq:tail-single-layer-input}, giving a contradiction.

Consequently, for all sufficiently large $N$ and all
$m\geq\lceil N^r\rceil$,
\begin{equation}
\label{eq:tail-inner-upper-bound}
\mathbb{E}e^{t_N(\lambda)X_m}
\leq
\mathbb{E}e^{t_N(\lambda)K_m}
\leq
\exp\left\{
c_\lambda s_N\frac{m}{N^r}
\right\}.
\end{equation}
The first inequality uses $0\leq M_{l,m}\leq K_m$
and $t_N(\lambda)>0$.

If $\lambda\leq0$, then $t_N(\lambda)\leq0$ and
$X_m\geq0$, so
\[
\mathbb{E}e^{t_N(\lambda)X_m}
\leq1
=
\exp\left\{
c_\lambda s_N\frac{m}{N^r}
\right\},
\]
where $c_\lambda=0$ in this case.
Thus, for every fixed $\lambda\in\mathbb{R}$,
all sufficiently large $N$, and all
$m\geq\lceil N^r\rceil$, we have
\[
\mathbb{E}e^{t_N(\lambda)X_m}
\leq
\exp\left\{
c_\lambda s_N\frac{m}{N^r}
\right\}.
\]

By independence of the inner partition and $\xi(N)$,
\[
H_N^X(\lambda)
=
\left.
\mathbb{E}e^{t_N(\lambda)X_m}
\right|_{m=\xi(N)}.
\]
Thus, for $M\geq1$ and all sufficiently large $N$,
\[
H_N^X(\lambda)\mathbf{1}_{\{W_N>M\}}
\leq
e^{c_\lambda s_NW_N}
\mathbf{1}_{\{W_N>M\}}.
\]
For any fixed $h>c_\lambda$, it follows that
\[
\begin{aligned}
\mathbb{E}\left[
H_N^X(\lambda)\mathbf{1}_{\{W_N>M\}}
\right]
&\leq
\mathbb{E}\left[
e^{c_\lambda s_NW_N}
\mathbf{1}_{\{W_N>M\}}
\right]\\
&\leq
e^{-(h-c_\lambda)s_NM}
\mathbb{E}e^{h s_NW_N}.
\end{aligned}
\]
Taking logarithms, dividing by $s_N$, and using
\eqref{eq:tail-outer-logmgf} proves
\eqref{eq:two-layer-tail-quantitative}.
Keeping $h>c_\lambda$ fixed and letting $M\to\infty$
yields \eqref{eq:two-layer-tail-control}.

Finally, the equivalent formulation follows from the tower
property, since $\{W_N>M\}$ is measurable with respect
to $\xi(N)$.
\end{proof}

%%%%%%%%%%%%%%%%%%%%%%%%%%%%%%%%%%%%%

\begin{theorem}[Two-layer MDP for $K_{\xi(N)}$]
\label{thm:two-layer-K}
Assume the conditions of Corollary~\ref{cor:MDP-newZ} and
Lemma~\ref{lem:uniform-small-tilt}. For a fixed $r\in(\alpha_2,1)$, let
\[
s_N=N^{\frac{r-\alpha_2}{1-\alpha_2}},
\qquad
b_N:=s_N\left(\frac{N^r}{s_N}\right)^{\alpha_1}.
\]
Then
$
\left\{
\frac{K_{\xi(N)}(\alpha_1,\theta_0)}{b_N}
\right\}_{N\ge1}
$
satisfies a moderate deviation principle (MDP) on
$\mathbb{R}$ with speed $s_N$ and good rate function
\[
I_K(x)=
\begin{cases}
(1-\alpha_1\alpha_2)
(\alpha_1\alpha_2)^{
\frac{\alpha_1\alpha_2}{1-\alpha_1\alpha_2}}
x^{\frac1{1-\alpha_1\alpha_2}},
&x\ge0,\\[1ex]
+\infty,&x<0.
\end{cases}
\]
\end{theorem}

\begin{proof}
For $a\in(0,1)$, write
\[
c_a=(1-a)a^{a/(1-a)},
\qquad
\rho=\alpha_1\alpha_2.
\]
By the outer-index result in
Corollary~\ref{cor:MDP-newZ},
$
W_N:=\frac{\xi(N)}{N^r}
$
satisfies an MDP with speed $s_N$ and good rate function
\[
I_W(y)=c_{\alpha_2}y^{1/(1-\alpha_2)},
\qquad y\geq0.
\]

Fix $\lambda\in\mathbb{R}$ and set
\[
t_N(\lambda)
=
\frac{\lambda s_N}{b_N}
=
\lambda\left(\frac{s_N}{N^r}\right)^{\alpha_1},
\qquad
q_\lambda=(\lambda_+)^{1/\alpha_1},
\]
where $\lambda_+=\max\{\lambda,0\}$.
Throughout the proof, write
$K_m=K_m(\alpha_1,\theta_0)$ and set $K_0=0$.
For $y\geq0$, define
\[
F_N(y;\lambda)
=
\frac{1}{s_N}
\log\mathbb{E}
\exp\left\{
t_N(\lambda)K_{\lfloor N^ry\rfloor}
\right\}.
\]
Since the inner partition is independent of $\xi(N)$,
\begin{equation}
\label{eq:K-conditional-logmgf}
\mathbb{E}e^{t_N(\lambda)K_{\xi(N)}}
=
\mathbb{E}
\exp\left\{
s_NF_N(W_N;\lambda)
\right\}.
\end{equation}

Put $\kappa=\frac{r-\alpha_2}{1-\alpha_2}.$
Since $0<\kappa<r$ and $s_N=N^\kappa$, we have
\[
\frac{s_N}{N^r}\to0,
\qquad
\frac{\log(N^r)}{s_N}\to0.
\]
Thus Lemma~\ref{lem:uniform-small-tilt}, applied with
$A_N=N^r$ and $a=\alpha_1$, gives, for every fixed $M>0$,
\[
\varepsilon_N(M)
:=
\sup_{0\leq y\leq M}
\left|F_N(y;\lambda)-q_\lambda y\right|
\longrightarrow0.
\]
Consequently,
\[
\begin{aligned}
&e^{-s_N\varepsilon_N(M)}
\mathbb{E}\left[
e^{s_Nq_\lambda W_N}
\mathbf{1}_{\{W_N\leq M\}}
\right]\\
&\qquad\leq
\mathbb{E}\left[
e^{s_NF_N(W_N;\lambda)}
\mathbf{1}_{\{W_N\leq M\}}
\right]\\
&\qquad\leq
e^{s_N\varepsilon_N(M)}
\mathbb{E}\left[
e^{s_Nq_\lambda W_N}
\mathbf{1}_{\{W_N\leq M\}}
\right].
\end{aligned}
\]
The large deviation upper and lower bounds for $W_N$
therefore imply
\begin{equation}
\label{eq:K-compact-laplace}
\begin{aligned}
&\lim_{N\to\infty}
\frac{1}{s_N}
\log\mathbb{E}\left[
e^{s_NF_N(W_N;\lambda)}
\mathbf{1}_{\{W_N\leq M\}}
\right]\\
&\qquad=
\sup_{0\leq y\leq M}
\left\{
q_\lambda y-I_W(y)
\right\}
=:J_M(\lambda).
\end{aligned}
\end{equation}
Indeed, the upper bound follows from the compact-set
Laplace upper bound. For the lower bound, one may use
open neighborhoods of points in $(0,M)$.
Continuity of $y\mapsto q_\lambda y-I_W(y)$ ensures
that its supremum over $(0,M)$ equals its supremum
over $[0,M]$.

On the other hand,
Lemma~\ref{lem:two-layer-tail-control} yields
\begin{equation}
\label{eq:K-tail-laplace}
\lim_{M\to\infty}\limsup_{N\to\infty}
\frac{1}{s_N}
\log\mathbb{E}\left[
e^{s_NF_N(W_N;\lambda)}
\mathbf{1}_{\{W_N>M\}}
\right]
=-\infty.
\end{equation}
Define
\[
J(\lambda)
=
\sup_{y\geq0}
\left\{
q_\lambda y-I_W(y)
\right\}.
\]
Since $I_W$ grows superlinearly, $J(\lambda)$ is finite,
and $J_M(\lambda)\uparrow J(\lambda)$ as $M\to\infty$.

The compact contribution in
\eqref{eq:K-compact-laplace} gives
\[
\liminf_{N\to\infty}
\frac{1}{s_N}
\log\mathbb{E}e^{t_N(\lambda)K_{\xi(N)}}
\geq J_M(\lambda).
\]
For the upper bound, split the expectation in
\eqref{eq:K-conditional-logmgf} over
$\{W_N\leq M\}$ and $\{W_N>M\}$, and use
\[
\log(u+v)\leq\log2+\max\{\log u,\log v\}.
\]
Combining \eqref{eq:K-compact-laplace} and
\eqref{eq:K-tail-laplace}, and then letting $M\to\infty$,
we obtain
\begin{equation}
\label{eq:K-variational-logmgf}
\lim_{N\to\infty}
\frac{1}{s_N}
\log\mathbb{E}e^{t_N(\lambda)K_{\xi(N)}}
=
\sup_{y\geq0}
\left\{
q_\lambda y-c_{\alpha_2}y^{1/(1-\alpha_2)}
\right\}.
\end{equation}

For every $u\geq0$, direct maximization gives
\[
\sup_{y\geq0}
\left\{
uy-c_{\alpha_2}y^{1/(1-\alpha_2)}
\right\}
=u^{1/\alpha_2}.
\]
Hence \eqref{eq:K-variational-logmgf} becomes
\begin{equation}
\label{eq:two-layer-K-logmgf}
\lim_{N\to\infty}
\frac{1}{s_N}
\log\mathbb{E}
\exp\left\{
\lambda s_N\frac{K_{\xi(N)}}{b_N}
\right\}
=
(\lambda_+)^{1/\rho}
=:\Lambda_K(\lambda).
\end{equation}

Since $0<\rho<1$, the function $\Lambda_K$ is finite
and continuously differentiable on $\mathbb{R}$.
The G\"artner--Ellis theorem therefore yields an LDP
with speed $s_N$ and good rate function
\[
I_K(x)
=
\sup_{\lambda\in\mathbb{R}}
\left\{
\lambda x-\Lambda_K(\lambda)
\right\}.
\]
For $x\geq0$, direct maximization gives
\[
I_K(x)
=
\sup_{\lambda\geq0}
\left\{
\lambda x-\lambda^{1/\rho}
\right\}
=
(1-\rho)\rho^{\rho/(1-\rho)}
x^{1/(1-\rho)}.
\]
For $x<0$, letting $\lambda\to-\infty$ gives
$I_K(x)=+\infty$.

Finally,
\[
b_N
=
N^{r\alpha_1+\kappa(1-\alpha_1)}
=
N^{\rho+(1-\rho)\kappa},
\qquad 0<\kappa<1.
\]
Thus $b_N/N^\rho\to\infty$ and $b_N/N\to0$, so this
LDP is a moderate deviation principle.
\end{proof}

\begin{remark}[Consistency with the coagulation identity]
Consider the special case $d=1$ and
$\theta_0=\theta/\alpha_2$, and write
$\rho=\alpha_1\alpha_2$.
By the Poisson--Dirichlet coagulation identity
\cite[Theorem~5.23]{Pitman2006},
\[
K_{\xi(N)}(\alpha_1,\theta_0)
\overset{d}{=}
K_N(\rho,\theta).
\]
The parameter condition $\theta_0>-\alpha_1$
ensures that $\theta>-\rho$.

Set $\kappa=\frac{r-\alpha_2}{1-\alpha_2}\in(0,1)$.
Since $s_N=N^\kappa$ and
$r=\alpha_2+(1-\alpha_2)\kappa$, we have
\[
b_N
=
N^{r\alpha_1}s_N^{1-\alpha_1}
=
N^\rho s_N^{1-\rho}.
\]
Thus, with $\beta_N=s_N^{1-\rho}$,
\[
\frac{\beta_N}{(\log N)^{1-\rho}}\to\infty,
\qquad
\frac{\beta_N}{N^{1-\rho}}\to0.
\]
The single-layer MDP in
\cite[Theorem~1]{FavaroFengGao2018}
therefore applies with speed
$\beta_N^{1/(1-\rho)}=s_N$ and rate function
\[
I(x)=
\begin{cases}
(1-\rho)\rho^{\rho/(1-\rho)}
x^{1/(1-\rho)}, & x\ge0,\\
+\infty, & x<0.
\end{cases}
\]
This agrees with Theorem~\ref{thm:two-layer-K}
and provides a consistency check in the
coagulation case.
\end{remark}

\begin{theorem}[Two-layer MDP for $M_{l,\xi(N)}$]
\label{thm:two-layer-mdp-M-CW}
Under the assumptions and notation of
Theorem~\ref{thm:two-layer-K}, fix an integer $l\geq1$
and set
\[
\rho=\alpha_1\alpha_2,
\qquad
p_l:=p_{\alpha_1}(l)
=
\frac{\alpha_1(1-\alpha_1)_{(l-1)\uparrow1}}{l!}.
\]
Then $\{M_{l,\xi(N)}/b_N\}_{N\geq1}$ satisfies an MDP
on $\mathbb{R}$ with speed $s_N$ and good rate function
\begin{equation}
\label{eq:two-layer-rate-M-CW}
I_l(x)
=
I_K\left(\frac{x}{p_l}\right)
=
\begin{cases}
\displaystyle
(1-\rho)\rho^{\rho/(1-\rho)}
p_l^{-1/(1-\rho)}
x^{1/(1-\rho)},
& x\geq0,\\[2mm]
+\infty,
& x<0.
\end{cases}
\end{equation}
\end{theorem}

\begin{proof}
Fix $\lambda\in\mathbb{R}$ and define
\[
t_N(\lambda)=\frac{\lambda s_N}{b_N},
\qquad
F_N^{(l)}(y;\lambda)
=
\frac{1}{s_N}
\log\mathbb{E}
\exp\left\{
t_N(\lambda)M_{l,\lfloor N^ry\rfloor}
\right\}.
\]
By Lemma~\ref{lem:uniform-small-tilt}, for every $M>0$,
\[
\sup_{0\leq y\leq M}
\left|
F_N^{(l)}(y;\lambda)
-
y(p_l\lambda_+)^{1/\alpha_1}
\right|
\longrightarrow0.
\]
Moreover, Lemma~\ref{lem:two-layer-tail-control}
applies to $X_m=M_{l,m}$. Thus the same conditioning
and truncation argument as in the proof of
Theorem~\ref{thm:two-layer-K} gives
\begin{equation}
\label{eq:two-layer-M-logmgf}
\begin{aligned}
&\lim_{N\to\infty}
\frac{1}{s_N}
\log\mathbb{E}
\exp\left\{
\lambda s_N\frac{M_{l,\xi(N)}}{b_N}
\right\}\\
&\qquad=
\sup_{y\geq0}
\left\{
y(p_l\lambda_+)^{1/\alpha_1}-I_W(y)
\right\}\\
&\qquad=
(p_l\lambda_+)^{1/\rho}.
\end{aligned}
\end{equation}

The limiting log-mgf is finite and continuously
differentiable on $\mathbb{R}$. The G\"artner--Ellis
theorem therefore yields the asserted MDP.
Since $p_l>0$, the change of variables $\mu=p_l\lambda$
in the Legendre transform gives
\[
\begin{aligned}
I_l(x)
&=
\sup_{\lambda\in\mathbb{R}}
\left\{
\lambda x-(p_l\lambda_+)^{1/\rho}
\right\}\\
&=
\sup_{\mu\in\mathbb{R}}
\left\{
\mu\frac{x}{p_l}-(\mu_+)^{1/\rho}
\right\}
=
I_K\left(\frac{x}{p_l}\right).
\end{aligned}
\]
The explicit expression follows from
Theorem~\ref{thm:two-layer-K}.
\end{proof}

%%%%%%%%%%%%%%%%%%%%%%%%%%%%%%%%%%%%%%

\section*{Statements and Declarations}

\begin{description}
\item[Competing interests.] The author declares that there are no competing interests.
\item[Funding.] Nian Yao was supported in part by the Natural Science Foundation of China under Grant 12071361 and Shenzhen Natural Science Fund (the Stable Support Plan Program) 20220810152104001. 
\item[Data availability.] Data sharing is not applicable to this article as no datasets were generated or analyzed.
\item[Code availability.] Not applicable.
\item[Author contributions.] The author contributed solely to this work.
\item[Acknowledgments.] The author gratefully acknowledges the hospitality of the Department of Mathematics and Statistics at McMaster University during her visit from March 2025 to February 2026. She is particularly grateful to Professor Shui Feng for his valuable guidance and insightful discussions. This work was carried out during the visit.
\end{description}

%%%%%%%%%%%%%%%%%%%%%%%%

\bibliographystyle{apalike}
\bibliography{reference-MDP}
\end{document}